\documentclass[]{amsart}
 \usepackage[]{amsmath, amsthm, amsfonts, amssymb}
 \usepackage[all]{xy}
 \usepackage{hyperref}

 \newcommand {\N} {{\mathbb N}}
 \newcommand {\C} {{\mathbb C}}
 \newcommand {\R} {{\mathbb R}}
 \newcommand {\Z} {{\mathbb Z}}
 
 \newcommand {\PP} {{\mathbb P}}
 \newcommand {\U} {{\mathcal U}}
 \newcommand {\F} {{\mathcal F}}
 \newcommand {\I} {{\mathcal I}}
 \newcommand {\E} {{\mathcal E}}

\newcommand {\dt} {{\bullet}}
\newcommand {\X} {{\mathcal X}}
\newcommand {\dB} {\underline{\Omega}}
\newcommand {\ddB} {\underline{\dB}}
\newcommand {\cL} {{\mathcal L}}
\newcommand {\cC} {{\mathcal C}}
\newcommand {\OO} {{\mathcal O}}

 \newtheorem{thm}[subsection]{Theorem}
\newtheorem{cor}[subsection]{Corollary}
\newtheorem{lemma}[subsection]{Lemma}
\newtheorem{prop}[subsection]{Proposition}

\DeclareMathOperator{\Spec}{Spec}

\begin{document}

\title{Frobenius Amplitude, Ultraproducts, and Vanishing on Singular Spaces}

\author{ Donu Arapura}
 \address{
Department of Mathematics\\
 Purdue University\\
 West Lafayette, IN 47907\\
 U.S.A.}
 \thanks{Author partially supported by the NSF}


\begin{abstract}
  A general Akizuki-Kodaira-Nakano vanishing theorem is proved for a
  singular complex projective variety by positive characteristic
  techniques. The passage to  characteristic zero is handled using ultraproducts.
\end{abstract}

\maketitle

When $X$ is a singular complex algebraic variety, Du Bois
\cite{dubois} defined a complex of sheaves $\dB_X^j$ which plays the
role of the sheaf of regular $j$-forms on a nonsingular variety. For
example, if $X$ is a projective variety, then $H^i(X,\C)$ decomposes
into a sum  $\oplus H^{i-j}(X,\dB_X^j)$ refining the classical Hodge decomposition.
Our goal is to prove a general vanishing theorem that for any complex
of locally free sheaves on a singular projective variety
$H^i(X,\dB^j_X\otimes \F^\dt)=0$ for $i+j\ge \dim X +\phi(\F^\dt)$,
where the Frobenius ampltitude $\phi(\F^\dt)$ refines the invariant
 introduced in \cite{arapura}. When combined with the bounds on $\phi$
 given in \cite{arapura,ara2}, we recover generalizations of
the  Akizuki-Kodaira-Nakano vanishing theorem
 due to Le Potier, Navarro Aznar and others. The vanishing theorem is
 deduced from an extension of the Deligne-Illusie decomposition
 \cite{deligne-ill} to Du Bois' complex. This  also leads to
 another proof of the Hodge decomposition in the singular case.

 In the first couple of sections, we re-examine the definition of 
 Frobenius amplitude. It is most natural over a field of
 characteristic $p>0$, and we do not change anything here. In our
 earlier work,  we extended the notion into
 characteristic $0$ by essentially taking the supremum of $\phi$ over
 all but finitely many mod
 $p$ reductions. In this paper, we relax the definition by
 replacing ``all but finitely many reductions'' by ``a large set of
 reductions''.  The result is potentially smaller (i.e. better) than
 before. The precise definition
 depends on making a suitable choice of what a large set of primes should
 mean. For the choice to be  suitable, we require that the collection
 of large sets forms a filter which is non principal in the
 appropriate  sense. We
 can  then recast the definition of Frobenius amplitude in terms of
 ultraproducts with respect to ultrafilters containing this filter.
 Since the use of  ultraproducts is not that common in algebraic
 geometry, we include a brief EGA-style treatment of
 them. We should point out that this discussion is not strictly necessary for the main
 result. Readers who prefer to do so can  jump to the
 final section and substitute the original definition for $\phi$
 whenever it occurs.

My thanks to the referees for pointing out various ambiguities in the original, and also
for suggesting the additional reference \cite{schout3}, which gives a
nice overview of ultraproducts in a related context.

\section{Ultraproducts of schemes}
Recall that a filter  on a set $S$ is a collection of nonempty subsets of  $S$
which is closed under finite intersections and supersets.
A property will be said to hold for 
{\em almost all} $s\in S$, with respect to a fixed filter $\F$, if the set of $s$
for which it holds lies in $\F$. 
 An ultrafilter 
is a  filter which is maximal with respect to inclusion. Equivalently,
an ultrafilter is a filter $\U$ such that for any $T\subseteq S$ either
$T\in \U$ or $S-T\in \U$.
For example, the set of all subsets
containing a fixed $s\in S$
is an ultrafilter. Such examples, called principal ultrafilters, are not particularly
interesting. If $S$ is infinite the set of cofinite subsets (complements of finite
sets) forms a non principal filter. By Zorn's lemma, this can be
extended to a non principal ultrafilter. Some results depend crucially
on the filter being an ultrafilter, so
to avoid confusion, we will reserve $\U$
exclusively for an ultrafilter in what follows

Suppose that $\F$ is a filter on $S$.
Given a collection of abelian groups (respectively
commutative rings)  $A_s$ indexed by $S$, the set $I_\F\subset \prod_s A_s$ of elements
which are zero for almost all $s$ forms a subgroup (respectively ideal).
The quotient  $\prod A_s/\F=\prod A_s/I_\F$ is their {\em filter
  product}. (This is commonly referred to as the reduced product, but
this would be too confusing when applied to commutative rings and  schemes.)
The filter product is called an {\em ultraproduct} when $\F=\U$ is an
ultrafilter. 
 Given an element of  $a\in \prod A_s/\F$ represented by a
sequence $(a_s)\in \prod A_s$, following \cite{schout3}, we will refer
to the elements $a_s$ as approximations to $a$.

\begin{prop}\label{prop:ultraprodfields}
If each $A_s$ is a field then any maximal ideal in $\prod A_s$  is given by $I_\U$
for some ultrafilter $\U$ on $S$. All prime are maximal. Suppose that $P$ is a property
expressible by  a set of first order sentences in the language of fields (for example
that the field is algebraically closed or has characteristic
$= n$). If $P$ is satisfied in $A_s$ for
almost all $s$ with respect to an ultrafilter $\U$, then $P$ is satisfied in $\prod A_s/\U$.
\end{prop}

\begin{proof}
  For $f\in \prod A_s$, let $z(f) =\{s\mid f_s=0\}$. 
  One has $z(fg)=z(f)\cup z(g)$ and $z(\alpha
  f+\beta g+\gamma fg)=z(f)\cap z(g)$ for appropriate coefficients
  depending on $f$ and $g$. From this, it follows that for any
  ideal $I\subset \prod A_s$, $F=z(I)$ is a filter. One can also check
 that if $F$ is a filter, then  $I_F=\{f\mid z(f)\in F\}$ is an ideal,
 such that $z(I_F)=F$ and $I_{Z(I)}=I$. Therefore, we obtain an order preserving
  bijection between the sets of ideals and filters. This proves
 the first statement, Suppose that $F$ is a filter which is not an ultrafilter. Then there exists a subset
  $T\subset S$ such that $T, S-T\in F$ \cite[chap I, lemma 3.1]{bell}
Let $\tau$ be the characteristic function of $T$
$$\tau_s= 
\begin{cases}
  1 &\text{if $s\in T$}\\
 0 & \text{otherwise}
\end{cases}
$$
Then it follows that $\tau, 1-\tau\in I_F$.  As
$\tau(1-\tau)=0$,  $I_F$ is not prime. This implies that prime ideals
necessarily arise from ultrafilters.
The last statement is a special
  case of {\L}os's theorem in model theory \cite[chap 5\S 2]{bell}.
\end{proof}

Filter products can be taken for other structures. For example
$\prod \N/\F$ will inherit the structure of a partially ordered
commutative semiring. By {\L}os's theorem, this satisfies the first
order Peano axioms if $\F=\U$ is an ultrafilter. In particular, it is totally ordered.
Under the diagonal embedding, $\N$ gets identified with an initial
segment of $\prod \N/\U$. The elements of the complement can be
thought of as infinitely large nonstandard numbers.

The basic references for scheme theory are  \cite{ega} and
\cite{hartshorne};  the first reference is a bit better for our
purposes, since it has less reliance on  the noetherian condition. 
To simplify the discussion, all schemes should be
assumed to be separated unless stated otherwise.
Given a collection of affine schemes $\{\Spec A_s\}_{s\in S}$, $\Spec (\prod_s
A_s)$ is their coproduct in the category of affine schemes,
although  {\em  not} in the category of schemes unless $S$ is finite. This is already clear
when $A_s$ are all fields,  $\coprod_s \Spec A_s=S$ while $\Spec (\prod_s A_s)$ is the set of
ultrafilters on $S$ by the previous proposition. In fact, as a space,
$\Spec (\prod_s A_s)$
 is  the Stone-\v{C}ech compactification of $S$.  This is a very
strange scheme from the usual viewpoint (it is not noetherian...), but this is
precisely the sort of construction we need. So it will be convenient
to extend this to  the category of all separated schemes.

\begin{prop}\label{prop:vee}
 There is a functor $\{X_s\}_{s\in S}\mapsto \bigvee_s X_s$ from the category of $S$-tuples of separated
 schemes to the category of separated schemes, 
such that it takes a family of open immersions to an open immersion and
$$\bigvee _s\Spec A_s\cong
 \Spec (\prod_s A_s).$$
 Moreover, there  are canonical morphisms
 $ X_s\to \bigvee X_s$ induced  by projection $\prod A_s\to A_s$ for
 affine schemes.  Given a collection of quasi-coherent
 sheaves $\F_s$ on $X_s$, we have a quasi-coherent sheaf
 $\bigvee_s\F_s$ on $\bigvee_s\X_s$ which restricts to $\F_s$ on each
 component $X_s$.
\end{prop}

\begin{proof}
Choose affine open covers $\{U_{sj} =\Spec A_{sj}\}_{j\in J_s}$ for each
$X_s$. After replacing each $J_s$ by  the
maximum of the cardinalities of $J_s$ and then allowing repetitions
$U_{sj_\alpha}=U_{s,j_{\alpha+1}}=\ldots$ if necessary, we can
assume that $J_s=J$ is independent of $s$. Then
$\bigvee_s X_s$ is obtained by gluing $\Spec (\prod_i A_{sj})$
together. A refinement of the open cover $\{U_{sj}\}$
can be seen to yield an isomorphic scheme. So the construction 
does not depend on this.
The projections $\Spec (\prod_s A_{sj})\to A_{sj}$ patch to yield
canonical maps $X_s\to \bigvee X_s$. 

Finally, given $\F_s=\widetilde{M_{sj}}$  on $\Spec A_{sj}$ we construct
$\F=\widetilde{\prod  M_{sj}}$ on the above cover, and then patch.
\end{proof}

We refer to $\bigvee X_s$ as the {\em affine coproduct}. Of course, we have
a morphism  $\coprod X_s\to \bigvee X_s$  from the usual coproduct,
but this usually is not an
isomorphism as we noted above.

Given  a scheme $Y$, 
let $\Sigma\subseteq Y$ be a set of points. Define
$$\beta (\Sigma) =\Spec\, \prod_{y\in \Sigma} k(y) $$
where $k(y)$ are the residue fields.
By proposition~\ref{prop:ultraprodfields}, the points of
$\beta(\Sigma)$ are necessarily closed, and they correspond to
ultrafilters on $\Sigma$.   
As a topological space, this can be identified with the Stone-\v{C}ech compactification 
of $\Sigma$.  The embedding $\Sigma\subset \beta(\Sigma)$ maps  a point
to the associated principal ultrafilter. Nonprincipal ultrafilters
give points on the boundary.

\begin{lemma}
If $Y$ is separated, there is a canonical morphism $\iota:\beta(Y)\to
Y$ of schemes induced by the canonical homomorphism
$$A\to \prod_{m\in \Spec A} k(m).$$
on any affine open set $Spec A\subset Y$.
\end{lemma}

\begin{proof}
  This is follows immediately by choosing an affine open cover.
\end{proof}

We will call a subset $\Sigma\subset Y$ {\em separating} if
$\beta(\Sigma)\to Y$  is injective on structure sheaves.
For example, $\Sigma=Y$ is separating when it is reduced, and the set of
closed points $Clsd(Y)$ is separating if in addition $Y$ is Jacobson.
The residue field  at an ultrafilter $\U$ on $\Sigma$ regarded as
point  in $\beta(\Sigma)$
 is none other than the ultrapoduct $k(\U)=\prod k(y)/\U$. Let
$$Y_\U = \Spec\, k(\U)\rightarrow \beta(\Sigma)\to \beta(Y)$$
be the corresponding map of schemes. 
Let us call an ultrafilter on  $\Sigma$, or the corresponding point of $\beta(\Sigma)$, {\em pseudo-generic}
if it contains all nonempty opens of $\Sigma$ with respect to the
topology induced from the Zariski topology on $Y$.  Such ultrafilters clearly
exist by Zorn's lemma. Pseudo-generic points will play the role of
generic points. The next lemma shows that such points do in fact
dominate the scheme theoretic generic point.

\begin{lemma}\label{lemma:genfilter}
Suppose that $Y$ is integral and separated and that
$\Sigma\subseteq Y$ is separating.
Let  $\U$ be a pseudo-generic ultrafilter. Then $k(\U)$ contains the function field $K(Y)$.
If $Y'\subset Y$ is a nonempty open subscheme then $Y_\U\cong Y_\U'$ where
$\U'= \{U\in \U\mid U\subseteq Y'\cap \Sigma\}$.
\end{lemma}

\begin{proof}
  We can reduce immediately to the case where $Y=\Spec\, A$, with $A$ a
  domain. By assumption the canonical map  $\psi:A\to \prod_{m\in \Sigma} A/m$ is injective.
As already noted, $L=(\prod_{m\in \Sigma} k(m))/\U$ is a field. An element
  $a\in A$ maps to zero in $L$ if and only if $U_1=\{m\mid a\in m\}\in
  \U$. On the other hand, the complement $U_2=\{m\mid a\notin m\}$ is
  open. If it is nonempty, then it would lie in $\U$ leading to  the
  impossible conclusion that $\emptyset\in \U$. Thus $U_2=\emptyset$
  which implies that $\psi(a)=0$ and therefore $a=0$. Thus
  $L$ contains $A$ and consequently its field of fractions $K(Y)$.
  
For the second part, we can check immediately that $\U'$ is an
ultrafilter on $\Sigma'=Y\cap \Sigma$, and that projection
$$\prod_{y\in \Sigma}k(y)\to \prod_{y\in \Sigma'} k(y)$$
is an isomorphism modulo $\U$ and $\U'$.
\end{proof}

Note that the field $k(\U)$ is usually much bigger
than $k(Y)$.

Given a collection of fields $k_s$ indexed by $S$, and $k_s$-schemes
$X_s$, we define their ultraproduct $\bigvee X_s/\U$ by the cartesian diagram
$$
\xymatrix{
 \bigvee X_s/\U\ar[r]\ar[d] & \bigvee X_s\ar[d] \\ 
 \Spec\, k(\U)=\Spec\, (\prod k_s/\U)\ar[r] & \Spec\, \prod k_s
}
$$
for any ultrafilter on $S$. It would more appropriate to call this the
ultra-coproduct, but we have chosen to be consistent with earlier
usage. The ultraproduct is clearly functorial in the obvious sense. 
Note that this construction makes sense even
when $\U$ is   replaced by a filter, and we will occasionally use it in the
more general setting. In this case the base need no longer  be the
spectrum of  a field.

We record the following which is an immediate consequence of the
construction.

\begin{lemma}\label{lemma:ultraprod}
If, for each $s$, $\{\Spec\, A_{sj}\}_{j\in J_s}$ is an affine open cover of $X_s$, then
$$ \Spec\, (\prod A_{sj_s}\otimes_{\prod k_s}\prod
k_s/\U)\cong
 \Spec\, (\prod A_{sj_s}/\U)$$
is an affine open cover of $\bigvee X_s/\U$ indexed by $\prod J_s/\U$.
\end{lemma}

It is easy to see from this, that our ultraproduct  coincides with the
identically named notion defined by Schoutens \cite[\S
2.6]{schoutens}.   Let us start by analyzing the topological properties.
By \cite[I, 1.1.1.10]{ega}, a separated scheme is
quasi-compact if and only if it admits a finite cover by open affine
schemes. Let say that the family $X_s$ is {\em uniformly
  quasi-compact} if each $X_s$ is covered by a fixed number of affine
schemes, where the number is independent of $s$. The point of the
definition is the following:

\begin{cor}\label{cor:quasicompact}
  The ultraproduct of  a uniformly
  quasi-compact family of schemes is again quasi-compact.
\end{cor}

\begin{proof}
Suppose that $X_s$ is covered by $N$ open affine sets
$\{\Spec A_{s,j}\}$, then the cover $\{\Spec \prod A_{sj}/\U\}$
is indexed by $\prod \{1,\ldots, N\}/\U$ which can be
identified (using {\L}os's theorem for example) with $\{1,\ldots, N\}$ under
the diagonal embedding.
\end{proof}

Suppose that  $f:X\to Y$ is a morphism of schemes.
Let $X_y$ denote the fibre over $y\in Y$, then
we have a commutative diagram
$$
\xymatrix{
 \bigvee_{y\in Y} X_y\ar[r]\ar[d] & X\ar[d] \\ 
 \beta(Y)\ar[r] & Y
}
$$
We thus get a morphism $\bigvee X_y \to \beta(Y)\times_Y X$
which is generally {\em not} an isomorphism. To see this,
let $Y= \Spec\, A$ and $X = \Spec A[x]$. Then the morphism corresponds
to the injective map of algebras
$$ (\prod k(m))[x]\to \prod (k(m)[x])$$
which is not surjective unless $\Spec A$ is finite.

Fix a separating set $\Sigma\subseteq Clsd(Y)$ and an
 ultrafilter  $\U$ on $\Sigma$.  We define the  {\em  ultra-fibre} over $\U$ by
$$X_\U = Y_\U\times_{\beta(\Sigma)} \bigvee X_y = \bigvee_{y\in \Sigma}
X_y/\U$$
where we recall that $Y_\U=\Spec k(\U)$. This construction can be
extended to filters as well.
It will be useful to view the ultra-fibre  as a kind of enhanced
fibre. We have a morphism to the usual fibre $\pi: X_\U \to
Y_\U\times_Y X$.
For a principal filter corresponding to $y\in Y$, it is easy to see that
this  gives an isomorphism $X_\U\cong X_y$. From now on, we will
assume  that  that $\U$ is pseudo-generic and that $Y$ is integral.
Then $Y_\U\times_Y X= Y_\U\times_{\Spec K(Y)}X_\eta$
is the generic fibre $X_\eta =\Spec\, K(Y)\times_Y X$ followed by an extension of scalars.
The map $\pi$  is usually not an isomorphism. The ultra-fibre carries
more structure. Any collection of endomorphisms $X_y\to X_y$ gives
rise to an endomorphism of $X_\U$. For example, if the residue fields
of the points in $\Sigma$ have finite characteristic, we get  a Frobenius morphism
$Fr:X_\U\to X_\U$ by assembling the usual $char(k(y))$-power Frobenius
maps on the components $X_y$.

The definition of   a (quasi)coherent sheaf on a ringed space can be
found in \cite[I, chap 1\S 5]{ega}. Roughly speaking, a sheaf of modules is coherent
if  it is locally finitely presented.
(NB: the definition in \cite{hartshorne} is only
correct for noetherian schemes.)
Given a collection of separated $k_s$-schemes $X_s$ and
 quasi-coherent sheaves $\F_s$ on $X_s$. We define the
quasi-coherent sheaf $\bigvee \F_s/\U$ as the pullback of $\bigvee
\F_s$ to  $\bigvee X_s/\U$, for any (ultra)filter $\U$.  Even if all
the  sheaves $\F_s$ are coherent, their ultraproduct need not be.
For a counterexample, we may take a family of locally free sheaves of
unbounded rank. However, the converse statement is true under a
finiteness condition.

\begin{lemma}\label{lemma:ultracoh} Suppose that  $\{X_s\}_{s\in S}$
  is a uniformly  quasi-compact family of schemes, and $\U$ is an
  ultrafilter on $S$.
Any coherent sheaf on $\bigvee X_s/\U$ is given by    $\bigvee
\F_s/\U$ for a collection of coherent sheaves $\F_s$.
\end{lemma}

We first start with a  general sublemma.

\begin{lemma}
Any coherent sheaf on $\Spec A$ is given by $\widetilde M$ with $M$ finitely presented.  
\end{lemma}

\begin{proof}
When $A$ is noetherian, this is a consequence of \cite[I, 1.1.5.1]{ega}. The
general case is proved the same way.  Any quasi-coherent sheaf
  $\F=\widetilde M$ for a uniquely determined module $M$ \cite[I, 1.1.4.1]{ega}.  Write  $M$ as a direct
  limit of finitely generated submodules $M=\varinjlim M_\lambda$. When $\F$ is coherent,
we have $\F=\widetilde M_\lambda$ for some $\lambda$ by \cite[I,
0.5.2.3]{ega}. Thus  implies that $M$ is finitely generated. Thus we have
a surjection $q:\OO_{\Spec A}^m\to \F$. Since $\ker q$ is coherent, we
see that it is also finitely generated, and this proves the result.
\end{proof}

\begin{proof}[Proof of lemma \ref{lemma:ultracoh}]
  First assume that $X_s=\Spec A_s$. Then the sublemma implies that  $\F$ is given by  a finitely
  presented $\prod A_s/\U$-module $M$. Fix a presentation matrix
  $(f_{ij,s})\in Mat_{m\times n}(\prod A_s/\U)$ for $M$. Then for each $s$, let $M_s$ be the cokernel of
  the approximation $(f_{ij,s})\in Mat_{m\times n}(A_s)$, 
and let $\F_s =\widetilde M_s$. Clearly $M$ is the ultraproduct
  of the corresponding modules, and so $\F\cong\bigvee \F_s/\U$.

For the general case, choose an open cover $\{U_{i,s}=\Spec A_{i,s}\}$ for
each $X_s$ by $N$ open sets. Let  $\F_i$ denote the restrictions of
$\F$ to $U_i=\bigvee_s U_{i,s}/\U$.
 Note that by corollary
\ref{cor:quasicompact}, or more accurately by its proof,
$\{U_i\}_{i=1,\ldots N}$ cover $\bigvee X_s$.
We can construct coherent sheaves $\F_{i,s}$ such that $\F|_{U_i}\cong \bigvee_s
\F_{i,s}/\U$ by the previous paragraph.  The identity maps $\phi_{ji}:\F_i|_{U_{ij}}=
\F_j|_{U_{ij}}$ can be approximated by maps  $\phi_{ji,s}:
\F_{i,s}|_{U_{ij,s}}\to  \F_{j,s}|_{U_{ij,s}}$. Using {\L}os's theorem, we can see that
$\phi_{ij,s}\phi_{ji,s}=id$ and  the cocycle identity
$\phi_{ij,s}=\phi_{ij,s}\phi_{ij,s}$ holds for almost all $s$.
For these values of $s$, we can glue $\F_{i,s}$ together using  with
$\phi_{ij,s}$ to form a coherent sheaf $\F_s$ such that
$\F_s|_{U_{i,s}}\cong \F_{i,s}$ \cite[I, chap 1 3.3.1]{ega}.
For the remaining $s$, we can simply take $\F_s=0$. With these choices,
the lemma is clearly satisfied.  
\end{proof}

\begin{cor}\label{cor:ultracoh}
  A coherent sheaf on the ultra-fibre of a projective morphism  is
  an ultraproduct of coherent sheaves on the fibres.
\end{cor}

\begin{proof}
  The fibres are uniformly quasi-compact.
\end{proof}

\begin{lemma}\label{lemma:Hiultra} Given a filter $\cL$, 
set $X= \bigvee X_s/\cL$ and $\F= \bigvee \F_s/\cL$.  Then
$H^i(X,\F)\cong \prod H^i(X_s,\F_s)/\cL$.
\end{lemma}

\begin{proof}
Choose affine open covers $\{U_{i,s}\}$ for each $X_s$. Then
we can compute  $H^i(X_s,\F_s)$
as the $i$th cohomology of the \v{C}ech complex 
$C_s^\dt=\check{C}^\dt(\{U_{i,s}\},\F_s)$. Similarly
 $H^i(X,\F)$ is the cohomology of
$$C^\dt
=\check{C}^\dt(\bigvee_s U_{i,s}/\cL, \bigvee_s \F_s/\cL)= \prod_s
C^\dt_s\otimes_{\prod k_s} \prod k_s/\cL $$
Since modules over a product of fields are flat, we can write
\begin{eqnarray*}
  H^i(C^\dt) &\cong& H^i(\prod C_s^\dt)\otimes \prod k_s/\cL\\
            &\cong& \prod H^i( C_s^\dt)\otimes \prod k_s/\cL\\
            &\cong&  \prod H^i(X_s,\F_s)/\cL
\end{eqnarray*}
\end{proof}

The cohomology groups $H^i(X,\F)$ may be infinite dimensional,
even when the sheaves $\F_s$ are coherent and the schemes are proper.
However, we can assign a generalized dimension $\dim H^i(X_,\F_s)\in \prod
\N/\U$.

Let $f:X\to Y$ be a morphism to an integral
scheme with  $\Sigma\subset Y$ separating.  Suppose that $\U$ is a
pseudo-generic ultrafilter. Then we have a
canonical map $\pi':X_\U\to X_\eta$ to the generic fibre.
 
\begin{lemma}
  Suppose that $f:X\to Y$ is projective, and $Y$ is noetherian.
If $\F$ is a coherent sheaf on $X_\eta$, then $H^i(X_\U,{\pi'}^*\F)\cong
H^i(X_\eta,\F)\otimes k(\U)$.
\end{lemma}

\begin{proof}
After shrinking $X$ and $Y$ if necessary, we can assume that $\F$ is the
restriction of a sheaf $\F'$  on $X$ and that $Y=\Spec\, A$.
Thanks to the semicontinuity theorem, c.f. \cite[III 12.11]{hartshorne},
by shrinking further, we can assume that the cohomology of $\F$ commutes with base
change which  means that $H^i(X,\F')$ is a free
$A$-module such that $H^i(X,\F')\otimes k(y)\cong H^i(X_y,\F'|_{X_y})$
for all $i$ and all (not necessarily closed) $y\in Y$.  By lemma
\ref{lemma:Hiultra},
\begin{eqnarray*}
H^i(X_\U,{\pi'}^*\F) &\cong& \prod H^i(X_y,\F'|_{X_y})/\U\\
                    &\cong&  \prod H^i(X,\F')\otimes_A k(y)/\U\\
                    &\cong&  H^i(X,\F')\otimes_A k(\U)\\
                   &\cong&  H^i(X, \F')\otimes_A K(Y)\otimes_{K(Y)} k(\U)\\
                   &\cong&  H^i(X_\eta, \F)\otimes_{K(Y)} k(\U)
\end{eqnarray*}

\end{proof}

 Let us call a coherent sheaf $\F$ on $X_\U$ {\em standard}
if it isomorphic to $\F'_\U:=\pi'^*\F'$ for some coherent sheaf $\F'$
on $ X_\eta$, where $\pi':X_\U\to X_\eta$ is the canonical map.

\begin{cor}\label{cor:standcohomfinite}
  A standard coherent sheaf on $X_\U$ has finite dimensional
  cohomology as a $k(\U)$-vector space.
\end{cor}

This corollary is not true for arbitrary coherent sheaves. For any
nonstandard natural number
 $N=(N_s)\in\prod \N/\U$, we can define the line bundle $\OO_{\PP^n_\U}(N) = \bigvee_s
 \OO_{\PP^n_s}(N_s)/\U$. Then $H^0(\OO_{\PP^n_\U}(N))$ is infinite
 dimensional in general.

A map of standard sheaves will be called standard if it is the pullback
of a map of sheaves on $X_\eta$. The category of standard sheaves and
maps is equivalent to the category of coherent sheaves on $X_\eta$ thanks to:

\begin{thm}[Van den Dries-Schmidt]
If $X\to Y$ is locally of finite type and $\U$ an ultrafilter, 
then $\pi: X_\U \to Y_\U\times_Y X$ is faithfully flat.
\end{thm}

\begin{proof}
  Since $\pi$ is evidently affine, this follows from the version
  given in \cite[3.1]{schout1}.
\end{proof}

Standard coherent ideal sheaves on projective space can be described quite
explicitly. 
The ring  $\prod k_s[x_0,\ldots x_n]/\U$ is graded by the monoid $\prod \N/\U$.
A  finitely generated homogeneous ideal $I\subset \prod k_s[x_0,\ldots x_n]/\U$ with
respect to this grading
determines a family of  homogeneous ideals $I_s\subset k_s[x_0,\ldots
x_n]$ such that $I=\prod I_s/\U$ \cite[2.4.12]{schout3}.
We can  form the associated  coherent sheaf $\I=\coprod \tilde I_s/\U$ on
$\PP^n_\U$.  Let us say that an element $(f_s)$ of $\prod k_s[x_0,\ldots
x_n]/\U$ has finite degree if there exists $d\in \N$ such that $\deg
f_s\le d$ for almost all $s$.

\begin{lemma}
$\I$ is standard coherent if $I$ is generated by a finite set of
elements with finite degrees. 
\end{lemma}
 
\begin{proof}
  Observe that we have an embedding
 $$(\prod k_s/\U)[x_0\ldots x_n]\subset \prod (k_s[x_0,\ldots
 x_n])/\U$$
under which elements on the left can be identified with
finite degree  elements. 
Thus the generators of $I$ are  polynomials. Therefore  $I$ is the
extension of 
$J= I\cap (\prod k_s/\U)[x_0\ldots x_n]$ to the bigger ring, and
the same goes for its localizations. This implies that $\I$ is
the pullback of the ideal sheaf associated to $J$.
\end{proof}

As an easy application of all of this, we show that  the cohomological
complexity  of a homogeneous  ideal, as measured by the
Castelnuovo-Mumford regularity,  can be bounded by a function of the degrees of its
generators. Although such results 
can be obtained more directly with effective bounds \cite{bayer, laz}, the proof here is
quite short.  For  other bounds in ideal theory obtained in the same
spirit, see \cite{van, schout3}. 
Given an ideal sheaf $\I$ on $\PP^n_k$, let $I=\oplus
\Gamma(\I(i))$ denote the corresponding ideal
and $d(I)$ the smallest integer such that
$I$ is generated
by homogeneous polynomials of degree at most $d(I)$.

\begin{lemma}[Bayer-Mumford]
Given $d,  n, i,m$ there exists a constant $C$ such
for any field $k$ and any ideal sheaf $\I$ on $\PP^n_k$ with
$d(\I) =d$, we have
$h^i(\PP_k^n, \I(m))< C$. In particular, the regularity 
 of $\I$ is uniformly bounded by a constant
depending only on $d(I)$ and $n$.
\end{lemma}

\begin{proof}
  Suppose the lemma is false. Then there is an infinite sequence of
  examples $\I_s,k_s$ such that $d(I_s)=d$ but $h^i(\I_s(m))\to
  \infty$. Therefore $\I = \bigvee \I_s/\U$ will have infinite
  dimensional $i$th  cohomology for any non principal ultrafilter $\U$. 
Let $N= \binom{n+d+1}{d+1}$. By allowing repetitions if necessary, for
each $s$ we can list generators
  $f_{1,s},\ldots f_{N, s}\in k_s[x_0,\ldots x_n]$ with degrees $\le d$ for
  the ideals $I_s$ corresponding to $\I_s$. Then the sequences $(f_{i,s})$ generate the ideal
  $I$ corresponding to $\I$. By the previous lemma $\I$ is
  standard. This implies, by corollary \ref{cor:standcohomfinite} , that the cohomology is finite dimensional,
  which  is a contradiction.
\end{proof}

\section{$F$-amplitude}

For the remainder of this paper, we fix a 
filter  $\cL$ on the set of prime numbers $\Sigma$ such that for any
$p\in \Sigma$, there exists $L\in \cL$ not containing $p$. The last
condition ensures that  any ultrafilter
containing $\cL$ is necessarily  non principal.
The elements of $\cL$ are the large sets of primes in the
introduction. 
We could take for $\cL$ the collection of cofinite subsets, 
or the filter generated by  complements of subsets of zero Dirichlet density.
Let $\OO_\Sigma=\prod \bar F_p$ be the product of algebraic closures of finite
fields. The ultraproduct  $k(\U)=\OO_\Sigma/\U$, for any $\U\supset \cL$, is an algebraically closed field of
characteristic zero with  cardinality $2^{\aleph_0}$.  Therefore there
is a noncanonical isomorphism  $k(\U)\cong \C$ which we fix for the
discussion below.

Suppose that $k$ is a field of  characteristic $0$. We can assume
without essential loss of generality that it is embedable into $\C$. 
Let $A(k)$ be the set of
finitely generated $\Z$-algebras contained in $k$. 
For each $A\in A(k)$, choose a separating family (defined previously) of maximal ideals
$m_p\in Max(A)$ with embeddings $A/m_p\subset \bar F_p$.  We assume that these choices are
compatible with the  inclusions $A_1\subseteq A_2$
(the existence of such compatible family is straightforward).
Given an algebraic variety $X$ (with a coherent sheaf $\F$) defined
over $k$, a thickening of $X$ (and $\F$) over $A\in A(k)$ is a flat
morphism $\X\to \Spec\, A$ (with an $A$-flat coherent sheaf $\tilde \F$) 
such that  $X\cong \Spec\, k\times_{\Spec  A}\X$ (and $\F$ is the
restriction of $\tilde \F$). A more detailed discussion of thickenings
and related issues can be found in \cite{arapura}. 
For any filter $\U\supseteq \cL$, we can form the
ultra-fibre $\X_\U$ after identifying
$\Sigma$ with the set of $m_p$. Since this is independent of the
thickening, we denote it by  $X_\U$. Ditto for $\F_\U$.  We will
assume that $\U$ is pseudo-generic. As explained
earlier, there is  a map
$\pi:X_\U\to X$ (such that $\F_\U$ is the pullback of $\F$).
Given $N=(N_p)\in \prod \N/\U$, let $Fr^N =X_\U\to X_\U$ be the morphism given by
the $p^{N_p}$th power Frobenius on $\X_{m_p}$.

We recall the original definition of Frobenius or $F$-amplitude from
\cite{arapura}. We will denote it by $\phi_{old}$ to
differentiate it from a variant $\phi$ defined below.
Given a locally free sheaf $\F$ on a variety $X$
defined over a field of characteristic $p>0$,
$\phi_{old}(\F)$ is the smallest natural number $\mu$ such that
for any coherent $\E$,
$$H^i(X,Fr^{N*}(\F)\otimes \E) = 0$$
for $i> \mu$ and $N\gg 0$. In this case, we set
$\phi(\F)=\phi_{old}(\F)$.
 In characteristic $0$, $\phi_{old}$
was defined using reduction modulo $p$:
$$\phi_{old}(\F) = \min_{(\X,\tilde\F)}(\max_{m} \phi(\tilde
\F|_{\X_m}))$$
where we maximize over all closed fibres of a thickening $(\X,\tilde
\F)$ of $(X,\F)$, and then minimize over all thickenings.  Basic
properties including finiteness can be found in \cite{arapura}.
The idea is take the worst case of $\phi$ among all fibres of the best
possible thickening. It is easy to see that for any  thickening,  
$$\phi_{old}(\F)\ge \phi(\tilde \F|_{\X_{m_p}})$$
for  all but finitely $p$ in $\Sigma$.
We  {\em redefine  Frobenius amplitude} in characteristic $0$ as 
  the smallest integer $\mu$ for which
$$\mu\ge \phi(\tilde \F|_{\X_{m_p}})$$
holds for almost all $p$ with respect to $\cL$.

\begin{lemma}
  \begin{enumerate}
\item[]
  \item For any locally free sheaf $\F$, we  have $\phi(\F)\le
    \phi_{old}(\F)$.
\item   $\phi(\F)$ is the smallest integer such
that for any  coherent sheaf of the form $\E=\bigvee \E_s/\cL$ on $\X_\cL$, there exists
$N_0\in \prod \N/\cL$ such that
$$H^i(X_\cL,Fr^{N*}(\F)\otimes \E) = 0$$
for $i>\phi(\F)$ and $N\ge N_0$. (We are suppressing $\pi^*$ above to
simplify notation.)

\item  $\phi(\F)$ is the smallest integer such
that for any ultrafilter $\U\supset \cL$ and any coherent
sheaf $\E$ on $\X_\U$, there exists $N_0\in \prod \N/\U$ such that
$$H^i(X_\U,Fr^{N*}(\F)\otimes \E) = 0$$
for $i>\phi(\F)$ and $N\ge N_0$. 
  \end{enumerate}
\end{lemma}

\begin{proof}
  (1) is immediate from the definition. (2) follows from lemma
   \ref{lemma:Hiultra}.
For (3), it is enough apply corollary \ref{cor:ultracoh} and  observe that for any family of vector spaces $V_p$,
$$\prod V_p/\cL = 0 \Leftrightarrow \prod V_p/\U = 0,\, \forall
\U\supset \cL$$
\end{proof}

We use the lemma to extend this notion  to a bounded complex of
coherent sheaves $\F^\dt$: 
 $\phi(\F^\dt)$ is the smallest integer such
that for any coherent sheaf $\E$ on $\X_\U$,
$$H^i(X_\U,\mathbb{L}Fr^{N*}(\F)\otimes^{\mathbb{L}} \E) = 0$$
for $i>\phi(F)$ and $N\gg 0$. 
Note that $Fr^N\circ \pi:X_\U\to X$ need not be flat when $X$ is
singular,  so to get a reasonable notion we are forced to take derived functors.
The following is immediate. 

\begin{lemma}
  For any distinguished triangle 
$$\F_1^\dt\to \F_2^\dt\to \F_3^\dt\to \F_1^\dt[1]$$
$\phi(\F_2^\dt)\le \max(\phi(\F_1^\dt),\phi(\F_3^\dt))$
\end{lemma}

\section{Frobenius split complexes}

Suppose for the moment that $X$ is a  scheme in
characteristic $p>0$ or an ultra-fibre, so that $X$ possesses a
Frobenius morphism $Fr$.
Let  $(\cC^\dt, F)$ be  a  bounded filtered complex of sheaves on $X$
with a finite filtration.
By a {\em Frobenius splitting } of the complex, we mean 
a diagram of quasi-isomorphisms
$$\bigoplus_i Gr_F^i \cC^\dt \stackrel{\sigma_1}{\longrightarrow} K^\dt
\stackrel{\sigma_2}{\longleftarrow}  Fr_* \cC^\dt
$$
or equivalently a representative for an isomorphism
$$\sigma:\bigoplus_i Gr_F^i \cC^\dt \cong Fr_* \cC^\dt
$$
in the derived category. 
A filtered complex is called Frobenius split 
possesses a Frobenius splitting. Although the terminology is
convenient in the present context, we warn
the reader  that  it conflicts slightly
with the standard notion of a  Frobenius split variety.
We make the collection of filtered complexes with splittings into
a category with morphisms given by  a morphism of filtered
complexes $(\cC_1,F_1)\to (\cC_2, F_2)$ together with a compatible
commutative diagram
$$
\xymatrix{
 \bigoplus_i Gr_F^i \cC_1^\dt\ar^{\sim}[r]\ar[d] & K_1\ar[d] & Fr_*\cC_1^\dt\ar_{\sim}[l]\ar[d] \\ 
 \bigoplus_i Gr_F^i \cC_2^\dt\ar^{\sim}[r] & K_2 & Fr_*\cC_2^\dt\ar_{\sim}[l]
}
$$
When $(\cC^\dt, F)$ is defined on a variety $X$ over a field of
characteristic zero, a Frobenius splitting will mean a Frobenius splitting of
its pullback to $X_\cL$.

The obvious question is how do Frobenius split complexes arise in nature?
In answer, we propose the
following vague  slogan: {\em Complexes $(\cC^\dt,F)$ arising from the
Hodge theory of varieties in
characteristic zero, with $F$ corresponding to the Hodge filtration,
ought to  be Frobenius split.}  Since the objects of Hodge theory are
usually highly transcendental, we should qualify this by restricting
to complexes of geometric origin.
However, we prefer not to try to make this too precise, but instead to keep it
as guiding principle in the search for interesting examples.
We begin with the basic example due
to  Deligne and Illusie  \cite{deligne-ill}:

\begin{thm}[Deligne-Illusie]\label{thm:di}
  Let $X $ be a smooth variety with a normal crossing divisor
  $D$ defined over a perfect field of characteristic
  $p>\dim X$. Suppose  that $(X,D)$ lifts mod $p^2$. Then there is an isomorphism
$$\sigma_X:\bigoplus_i \Omega_X^i(\log D)[-i]\cong
Fr_*\Omega_X^\dt(\log D)$$
in the derived category which depends canonically on the mod $p^2$ lift of $(X,D)$.
\end{thm}

\begin{cor} If $(X,D)$ is as above, or defined over a field of
  characteristic $0$,
  the logarithmic de Rham complex  $\Omega_X^\dt(\log D)$ with its stupid filtration,
  $F^i=\Omega_X^{\ge i}(\log D)$ is Frobenius split. 
\end{cor}

The functoriallity statement given in the theorem is not good enough
for our purposes. 
 The isomorphism $\sigma_X$ is realised  explicitly as a map $\tilde
\sigma_X$ from $\oplus \Omega_X^i(\log D)[-i]$ to a sheafified \v{C}ech complex
$\check{\mathcal C}(\{U_j\},Fr_*\Omega_X^\dt(\log D))$ 
with respect to an affine open cover of $X$. In addition to
the cover, it depends on 
mod $p^2$ lift of $(X,D)$ and mod $p^2$ lifts of $Fr|_{U_j}$. It is clear
that given any morphism $f:X_1\to X_2$, with $D_1\supseteq f^{-1}D_2$, which lifts mod $p^2$,
that compatible choices can be made. Then from the formulas in
\cite{deligne-ill}, we see that we get a morphism of Frobenius split complexes
extending the natural map $\Omega_{X_2}^\dt(\log D_2)\to f^*\Omega_{X_1}^\dt(\log
D_1)$.

An additional example of a Frobenius split complex, consistent with the
earlier principle, is provided by a theorem
of Ilusie \cite[4.7]{illusie} which implies:

\begin{prop}
  Let $f:X\to Y$ be a proper semistable map with discriminant $E\subset
  Y$ defined over a  field $k$ of characteristic $p\gg 0$ which lifts
  mod $p^2$. Let
  $H=R^if_*\Omega^\dt_{X/Y}(\log D)$ be a Hodge bundle with  filtration $F^j=
  R^if_*\Omega^{\ge j}_{X/Y}(\log D)$, where $D=f^{-1}E$. If $\nabla$
  denotes the Gauss-Manin connection, then the complex
$$H\stackrel{\nabla}{\to} \Omega_X^1(\log E)\otimes
H\stackrel{\nabla}{\to}\Omega_X^2(\log E)\otimes
H\stackrel{\nabla}{\to}\ldots $$
with filtration
$$F^j\stackrel{\nabla}{\to} \Omega_X^1(\log E)\otimes
F^{j-1}\stackrel{\nabla}{\to}\ldots$$
is Frobenius split.
\end{prop}

Further examples of Frobenius split complexes can be built from simpler
pieces using mappling cones. More generally, given a bounded complex 
$$(\cC^{\dt,0},F)\to (\cC^{\dt,1}, F)\to \ldots $$
of Frobenius split complexes, we can form the total complex
$$Tot(\cC^{\dt,\dt})^i = \bigoplus_{j+k=i} \cC^{jk}$$
in the usual way with filtration
$$F^pTot(\cC^{\dt,\dt})^i =\bigoplus_{j+k=i} F^{p}\cC^{jk} $$
Together with the diagram
$$Tot(\bigoplus Gr^\dt(\cC^{\dt\dt}))\to Tot(K^{\dt\dt})\leftarrow Fr_*Tot(\cC^{\dt\dt})$$
this becomes a Frobenius split complex.

Let us call a filtered complex $(\cC^\dt, F)$ coherent if it is a bounded
complex of quasi-coherent sheaves such that 
 the differentials are differential operators and $Gr_F\cC^\dt$ is quasi-isomorphic
to a complex of coherent sheaves with $\OO_X$-linear differentials.
For example, $(\Omega_X^\dt(\log D),  \Omega_X^{\ge i}(\log D))$ is
coherent. A slight refinement of the arguments of Deligne and Illusie yields:

\begin{thm}\label{thm:van0}
Let $X$ be complete variety in positive characteristic
 or the ultra-fibre of a complete variety in characteristic zero.
Suppose that $(\cC^\dt, F)$ is a coherent Frobenius split complex on $X$. Then
\begin{enumerate}
\item The spectral sequence
$$E_1^{ij}= H^{i+j}(X,Gr_F^i\cC^\dt) \Rightarrow H^{i+j}(X,\cC^\dt)$$
degenerates at $E_1$
\item For any bounded complex of
 locally free sheaves $\F^\dt$, 
$$H^i(X, Gr_F^j \cC^\dt\otimes \F^\dt) = 0$$ 
for any $j$ and $i> m + \phi(\F^\dt)$, 
where $m=\max\{\iota \mid \mathcal{H}^\iota(Gr_F^j\cC^\dt)\not= 0\}$.
\end{enumerate}
\end{thm}

\begin{proof}
Since $E_\infty$ is a subquotient of $E_1$, to prove  $E_1\cong
E_\infty$ it is enough to prove equality of dimensions.
 The morphism $Fr$ is affine, by definition in the first case and
 because it is an ultraproduct of affine
 morphisms in the second. 
Therefore $R^iFr_*Gr_F^j\cC^\dt=0$ for $i>0$, which implies  
$\R  Fr_*\cC^\dt= Fr_*\cC^\dt$. Thus 
$$H^j(X,\cC^\dt) \cong H^j(X,Fr_*\cC^\dt) \cong \bigoplus_i
H^{j}(X,Gr_F^i\cC^\dt)$$
which forces $\dim E_1=\dim E_\infty$ and proves (1).

  By definition of $\phi$
$$H^i(X, \mathcal{H}^k(Gr_F^j\cC^\dt)\otimes Fr^{N*}\F^\dt) = 0$$
for $i>\phi(\F^\dt)$, all $j$ and $N\gg 0$. So by a standard spectral
sequence argument,
$$H^i(X, Gr_F^j\cC^\dt\otimes Fr^{N*}\F^\dt) = 0$$
for $i> m + \phi(\F^\dt)$ and $N\gg 0$.
So (2) is consequence of the sublemma:

If $H^i(X, Gr_F^j\cC^\dt\otimes Fr^{*}\F^\dt) = 0$ for all $j$,
then  $H^i(X, Gr_F^j\cC^\dt\otimes \F^\dt) = 0$ for all $j$.

\begin{proof}
  The assumption forces $H^i(X, \cC^\dt\otimes Fr^{*}\F^\dt) = 0$.
On the other hand, the projection
formula and existence of a Frobenius splitting implies
  $$H^i(X,\cC^\dt\otimes Fr^*\F^\dt) \cong
  H^i(X,(Fr_*\cC^\dt)\otimes \F^\dt ) \cong \bigoplus_j
H^{i}(X,Gr_F^j\cC^\dt\otimes \F^\dt)$$
\end{proof}
This concludes the proof of the theorem.
\end{proof}

From this we recover the key degeneration of spectral sequence and
 vanishing theorems of 
\cite{deligne-ill}, \cite{illusie}, \cite{arapura} and \cite{ara2}

\section{Splitting of  the Du Bois complex}

Our goal is to prove a general Akizuki-Nakano-Kodaira type vanishing theorem for singular
varieties. The right replacement for differential forms in the Hodge
theory of such spaces was found by
Du Bois \cite{dubois}.
Given a complex algebraic variety $X$, Du Bois
 constructed a filtered complex $(\ddB_X^\dt, F)$ of sheaves, such that
\begin{enumerate}
\item The complex is unique up to filtered quasi-isomorphism. In other
  words, it is well defined in the filtered derived category $DF(X)$
\item There exists a map of complexes from the de Rham complex
with the stupid filtration $(\Omega_X^\dt,\Omega_X^{\ge p})$ to $(\ddB^\dt_X,F^p)$.
This is a filtered quasi-isomorphism when $X$ is smooth.
\item The complexes $\dB_X^i=Gr^i_F\ddB_X^\dt[i]$ give  well defined
  objects in the bounded derived category of coherent sheaves $D_{coh}^b(\OO_X)$ . (The
  shift is chosen so that $\dB_X^i=\Omega_X^i$ when $X$ is smooth.)

\item The associated analytic sheaves $\ddB_X^{\dt,an}$ resolve $\C$.
When $X$ is complete, the spectral sequence 
\begin{equation*}
E_1^{ab}= H^b(X,\dB_X^a)\Rightarrow H^{a+b}(X^{an},\C)  
\end{equation*}
degenerates at $E_1$ and abuts to the Hodge filtration for the
canonical mixed Hodge structure on the right.
\end{enumerate}
This can be refined  for pairs \cite[\S 6]{dubois}. If $Z\subset X$ is
a closed set with dense complement, there exists a filtered complex 
$(\ddB_X(\log Z), F)\in DF(X)$ 
such that $\dB_X^i(\log Z) = Gr_F^i\ddB_X(\log Z)[i]\in
D_{coh}^b(\OO_X)$ and there is a spectral sequence
\begin{equation}
  \label{eq:hodge2dubois2}
E_1^{ab}= H^b(X,\dB_X^a(\log D))\Rightarrow H^{a+b}((X-Z)^{an},\C)  
\end{equation}
which degenerates when $X$ is complete.

At the heart of the construction is cohomological descent (cf
\cite{deligne,gnpp,ps}), which is a refinement of \v{C}ech theory.
  Using resolution of singularities  one can construct  a diagram
$$
\xymatrix{
 \ldots\ar@<1ex>[r] \ar@<-1ex>[r]\ar[r] & X_1\ar@<1ex>[r]^{\delta_0}\ar[r]_{\delta_1} & X_0\ar[r] & X
}
$$
such that  $X_i$ are smooth,  the usual simplicial identities hold,
and cohomological descent is satisfied.
The last condition means that the
cohomology of any sheaf $\F$ on $X$ can be computed on $X_\dt$ as follows.
A simplicial sheaf is a collection of sheaves
$\F_i$ on $X_i$ with maps $\delta_j^*\F_i\to \F_{i+1}$. We define
$\Gamma(X_\dt,\F_\dt) =\ker[\delta_0^*-\delta_1^*:\Gamma(\F_0)\to \Gamma(\F_1)]$
and $H^i(X_\dt,\F_\dt) = R^i\Gamma(X_\dt,\F_\dt)$. If 
$\F_\dt$ is replaced by a  resolution by injective simplicial sheaves
$\mathcal{I}^\dt_{\dt}$ then $ H^i(X_\dt,\F_\dt)$
is just  the cohomology of the total complex 
$$ Tot(\Gamma(\mathcal{I}^\dt_{0})\to \Gamma(\mathcal{I}^\dt_{1})\to\ldots )$$
The pullback of $\F$ gives a simplicial sheaf $\F_\dt$ on $X_\dt$, and the
descent condition requires that $H^i(X,\F) \cong H^i(X_\dt,\F_\dt)$.
It is  important for our purposes to note that the diagram
$X_\dt$ can be assumed finite, in fact with the bound $\dim X_i\le
\dim X-i$, thanks to  \cite{gnpp}. Also if a proper closed set $Z\subset X$
is given, then one can construct a simplicial resolution such
that preimage  $Z_\dt$ of $Z$ on each $X_i$ is essentially a union of a divisor
with normal crossings (see \cite[5.21]{ps} for the precise conditions).

We recall the construction of Du Bois's complex. Choose a
smooth simplicial scheme $f_\dt:X_\dt\to X$ as above. Then
$(\Omega_{X_\dt}^\dt,\Omega_{X_\dt}^{\ge \dt})$ gives a filtered complex of simplicial sheaves on $X_\dt$.
By modifying the procedure for defining cohomology described above, 
we can form higher direct images for such objects. One then sets
\begin{equation}
  \label{eq:dubois}
(\ddB_X^\dt,F^{\dt}) = \R f_{\dt*}(\Omega_{X_\dt}^\dt, \Omega_{X_\dt}^{\ge \dt})  
\end{equation}
and in the ``log'' case
\begin{equation}
  \label{eq:dubois2}
(\ddB_X^\dt(\log Z),F^{\dt}) = \R f_{\dt*}(\Omega_{X_\dt}^\dt(\log
f_{\dt}^{-1}Z),\ldots )
\end{equation}
It follows that
$$ \dB_X^i = \R f_{\dt*}\Omega_{X_\dt}^i=Tot(
\R f_{0*}\Omega_{X_0}^i\to\R f_{1*}\Omega_{X_1}^i\to\ldots)  $$
In particular, from Grauert-Riemenschneider's vanishing theorem and
the dimension bound, we get an
elementary description of the top level
$$\dB_X^n = f_{0*}\Omega_{X_0}^n,\quad n=\dim X$$
In fact, this formula holds when $f_0$ is replaced by a resolution of
singularities \cite[p 153]{gnpp}.

In positive characteristic, de Jong's results \cite{dejong} 
on smooth alterations
can be used to construct a smooth simplicial scheme $X_\dt\to X$
satisfying descent. However, this is not good enough to guarantee a well
defined Du Bois complex. 
In our case, we can avoid these problems by applying \eqref{eq:dubois}
and \eqref{eq:dubois2}  to the mod $p\gg 0$ fibres of a thickening 
$\X_\dt\to \X\supset \mathcal{Z}$ of a simplicial resolution of complex varieties.
Equivalently, we can work with the ultra-fibres $X_{\dt,\U}\to
X_\U\supset Z_\U$. The following is suggested by the principle
enunciated in the last section.

\begin{thm}\label{thm:dubois}
 If $X$ is defined over a field of characteristic $0$, then
$(\ddB_X^\dt(\log Z),F) $ is Frobenius split.
\end{thm}

\begin{proof}
We have to show that
 $$\bigoplus_i \dB_{\X_p}^i(\log \mathcal{Z}_p)[-i]\cong Fr_*\ddB_{\X_p}^\dt(\log  \mathcal{Z}_p)$$
 for $p\gg 0$.
For $p$  large, $f:\X_{\dt,p}\to \X_p$ is a smooth simplicial scheme.
Then from    theorem \ref{thm:di} and the remarks following it, we
obtain an isomorphism 
$$\bigoplus_i \Omega_{\X_{\dt,p}}^i(\log\mathcal{Z}_p)[-i]\cong
Fr_*\Omega_{\X_{\dt,p}}^\dt(\log\mathcal{Z}_p)$$
of simplicial sheaves. Therefore
\begin{eqnarray*}
  \bigoplus_i \dB_{\X_p}^i(\log\mathcal{Z}_p)[-i] 
   &\cong& \bigoplus_i \R f_{\dt*}\Omega_{\X_{\dt,p}}^i(\log\mathcal{Z}_p)[-i]\\
   &\cong&   \R f_{\dt*} Fr_*\Omega_{\X_p}^\dt(\log\mathcal{Z}_p)\\
   &\cong&    Fr_* \R f_{\dt*}\Omega_{\X_p}^\dt(\log\mathcal{Z}_p)\\
    &\cong& Fr_*\ddB_{\X_p}^\dt(\log\mathcal{Z}_p)
\end{eqnarray*}

\end{proof}

As a corolloray we can reprove Du Bois' result.

\begin{cor}
  When $X$ is complete the spectral sequence (\ref{eq:hodge2dubois2}) 
degenerates.

\end{cor}

\begin{proof} Apply theorem~\ref{thm:van0}.
\end{proof}

\begin{cor}
 If $X$ is a complete complex variety and $\F^\dt$ a bounded complex of
 locally free sheaves, then
$$H^i(X,\dB_X^j(\log Z)\otimes \F^\dt) = 0$$ 
for $i+j> \dim X + \phi(\F^\dt)$. In particular, if $\F$  is a
$k$-ample vector bundle in
Sommese's sense, then $H^i(X,\dB_X^j(\log Z) \otimes \F)$
vanishes for $i+j\ge \dim X+rk(\F) +k$.
\end{cor}

\begin{proof}
The first statement follows from theorem~\ref{thm:van0}.
For the second, we can appeal to the estimates on $\phi$ proved in 
\cite[6.1]{arapura} and \cite[2.13, 5.17]{ara2}.
\end{proof}

The special case of the last result for ample line bundles is due to  
Navarro  Aznar \cite[chap. V]{gnpp} when $Z=\emptyset$, and  Kovacs \cite{kovacs}
in general.


\begin{thebibliography}{ABCD}


\bibitem[A1]{arapura} D. Arapura, {\em Frobenius amplitude and strong
  vanishing for vector bundles,} Duke Math. J. 121 (2004)


\bibitem[A2]{ara2} D. Arapura, {\em Partial regularity and amplitude},
Amer. J. Math. 128 (2006)
 
\bibitem[BM]{bayer} D. Bayer, D. Mumford, {\em What can be computed
in algebraic geometry}, Comp. Alg. Geom and Comm. Alg., Cambridge (1993)

\bibitem[BS]{bell} J. Bell, A. Slomson, {\em Models and ultraproducts,}
  North-Holland (1969)



\bibitem[De]{deligne} P. Deligne, {\em Theorie de Hodge III},
  Publ. Math. IHES 44 (1974)


\bibitem[DI]{deligne-ill} P. Deligne, L. Illusie,
{\em Relevetments modulo $p^2$  et decomposition du
complexe de de Rham}, Inv. Math. 89 (1987), 247-280

\bibitem[DS] {van} L. van den Dries, K. Schmidt, {\em Bounds in theory of polynomial rings.
A nonstandard approach,} Invent. Math (1984)


\bibitem[Du]{dubois} P. Du Bois, {\em Complexe de de Rham filtr\'e
    d'une vari\'et\'e singuli\`ere.} Bull. Soc. Math. France, 109 (1981)




\bibitem[EGA]{ega}  A Grothendieck, J. Deudonn\'e, {\em  \'Element de
    G\'eometrie Alg\'ebriques} I   Publ. Math. IHES 4; II
  Publ. Math. IHES 8; III   Publ. Math. IHES 11,  17; IV Publ. Math. IHES  20, 24, 28, 32 (1964-67)

\bibitem[GNPP]{gnpp} F. Guillen, V. Navarro Aznar, P. Pascual-Gainza,
  F. Puerta.
{\em Hyperr\'esolutions cubiques et descente cohomologique.} 
  Lecture Notes in Mathematics, 1335. Springer-Verlag, (1988) 

\bibitem[H]{hartshorne} R. Hartshorne, {\em Algebraic geometry},
  Springer-Verlag (1977)


\bibitem[J]{dejong} A. J. de Jong, {\em Smoothness, semi-stability and
    alterations.}  . Publ. Math. IHES 83 (1996)



\bibitem[I]{illusie} L. Illusie,  {\em R\'eduction semi-stable et
    d\'ecomposition de complexes de de Rham \`a coefficients.} 
Duke Math. J. 60 (1990)

\bibitem[Kv]{kovacs} S. Kovacs, {\em Logarithmic Vanishing Theorems and 
Arakelov-Parshin Boundedness for Singular Varieties } Compositio
Math 131 (2002)


\bibitem[L]{laz} R. Lazarsfeld, {\em Positivity in algebraic geometry I},
  Springer-Verlag (2004)

\bibitem[PS]{ps} C. Peters, J. Steenbrink, {\em Mixed Hodge
structures}, Springer-Verlag (2008)

\bibitem[S1]{schout1} H. Schoutens, {\em Nonstandard tight closure for
    affine $\C$-algebras}, Manuscripta Math. 111 (2003)

\bibitem[S2]{schoutens} H. Schoutens, {\em Log-terminal singularities and vanishing
theorems via non-standard tight closure,} J. Alg. Geom. 14 (2005)

\bibitem[S3]{schout3} H. Schoutens, {\em The use of ultraproducts in
    commutative algebra}, Lect. Notes Math 1999, Springer (2010)


\bibitem[St]{steenbrink} J. Steenbrink, {\em Vanishing theorems on
    singular spaces.} Differential systems and singularities. Ast\'erisque  130 (1985)

\end{thebibliography}
\end{document}